\documentclass[12pt]{article}
\usepackage[latin1]{inputenc}
\usepackage{amsfonts}
\usepackage{amsmath}
\usepackage{amssymb,amscd,latexsym,amsthm}
\usepackage{enumerate}
\usepackage[dvips]{graphicx}
\newtheorem{lemma}{Lemma}
\newtheorem{theorem}{Theorem}

\newtheorem{corollary}{Corollary}
\newtheorem{definition}{Definition}

\title{\vspace{-0.in}\parbox{\linewidth }{\footnotesize\noindent
} \\  \bf Symmetry, bifurcation and stacking of the central configurations of the planar $1+4$ body problem}

\author{
\\ Allyson Oliveira \\
\small {N\'ucleo de Forma\c c\~ao Docente, Universidade Federal de Pernambuco}
\\ \small{Caruaru-PE, CEP. 55002-970, Brazil e-mail: allyson.oliveira@ufpe.br }
}

\date{ }

\begin{document}
\maketitle

\author{ \ }

\begin{abstract}
In this work we are interested in the central configurations of the planar $1+4$ body problem where the satellites have different infinitesimal masses and two of them are diametrically opposite in a circle. We can think this problem as a stacked central configuration too. We show that the configuration are necessarily symmetric and the other sattelites has the same mass. Moreover we proved that the number of central configuration in this case is in general one, two or three and in the special case where the satellites diametrically opposite have the same mass we proved that the number of central configuration is one or two saying the exact value of the ratio of the masses that provides this bifurcation.
\end{abstract}


\maketitle

\section{Introduction}

In the N-body problem the central configurations are the configurations such that the total Newtonian acceleration of every body is equal to a constant multiplied by the position vector of this body with respect to the center of mass of the configuration. One of the reasons why central configurations are interesting is that they allow us to obtain explicit homographic solutions of the N-body problem i. e. motion where the configuration of the system changes its size but keeps its shape. They also arise as the limiting configuration of a total collapse. The papers dealing central configuration has focus in several points as finding example of particular central configuaration, giving the number of central configuration, its symmetry, stability, stacking properties, i. e. central configuration that contains others central configuration by dropping some bodies, etc. We can cite several references to reader about this theme; see e.g. \cite{hagihara},\cite{saari}, \cite{wintner}, \cite{saari2}.

This work deals with central configurations of the planar $5$-body problem, in the case where we have one dominant mass and $4$ infinitesimal masses, called satellites, on a plane. The planar $1+n$ body problem was treated by Maxwell \cite{maxwell} trying to construct a model for Saturn's rings. Many other contribution there are in the Literature for this problem. Considering satellites with equal masses, Casasayas, Llibre and Nunes \cite{llibre1} improving a previous result of Glen Hall \cite{hall} proving that the regular polygon is the only central configuaration if $n \geq e^{73}$. Cors, LLibre and Oll\'e \cite{llibre2} obtain numerically evidences that there is only one central configuration if $n\geq 9$ and that every central configuration is symmetric with respect to a straight line. Moreover they proved that there are only 3 symmetric central configurations of the $1+4$ body problem. Albouy and Fu \cite{albouy} proved that all central configurations of four identical satellites are symmetric which settles the question in the case $1+4$.

In the recent work \cite{oliveira}, Oliveira and Cabral worked with stacked planar central configuration of the $1+n$ body  problem in two cases: adding one satellite in a central configuaration with different satellites and adding two satellites considering equal all infinitasimal masses. Renner and Sicardy \cite{sicardy} obtained results about giving a configuration of the coorbital satellites, find the infinitesimal masses making it a central configuration. They also studied the linear stability of this configuaration. Corbera, Cors and Llibre \cite{llibre3} considering the $1+3$ body problem, found two different classes exhibiting symmetric and nonsymmetric configurations. And when two infinitesimal masses are equal, they provide evidence that the number of central configurations varies from five to seven. 

In this work we are interested in the central configurations of the planar $1+4$ body problem where the satellites have diferent masses. The satellites lie on a circle centered at the big mass and we treated the cse where two of them are diametrically opposite and nonconsecutive in the circle: for example we can suppose the big mass, the first satellite and the third one are collinear. The reader can observe that this problem is a stacked central configuaration too since the collinear $1+2$ configuration is a central one.

\section{Preliminaries}\label{preliminaries}

We develop in this section the formulation of this well known problem. Basically it references our previous work \cite{oliveira}. Consider $N$ pontual masses, $m_1,...,m_N$, in $\mathbb{R}^2$ subject to their mutual Newtonian gravitational interaction. Let $M=diag\{m_1,m_1,...,m_N,m_N\}$ be the matrix of masses and let $q=(q_1,...,q_N), q_i\in \mathbb{R}^2$ be the position vector. The equations of motion in an inertial reference frame with origin at the center of mass are given by 
$$M\ddot{q} = \frac{\partial V}{\partial q},$$
where $V(q_1,...,q_N)=\displaystyle\sum_{1\leq i<j\leq N}\frac{m_im_j}{\|q_i-q_j\|}$ is the Newtonian potential.

A non-collision configuration $q=(q_1,...,q_N)$ with $\sum_{i=1}^Nm_iq_i =0$ is a {\it central configuration} if there exists a positive constant $\lambda$ such that
$$M^{-1}V_q = \lambda q.$$

We consider the planar $N=1+n$ body problem, where the big mass is equal to 1 with position $q_0=0$. The remaining $n$ bodies with positions $q_i$, called satellites, have masses $m_i =\mu_i\epsilon, i=1,..,n$, where $\mu_i\in \mathbb{R}^{+}$ and $\epsilon > 0$ is a small parameter that tends to zero.  So we say that $(q_1,...,q_n)$ is a planar central configuaration of $1+n$ bodies if 
$$\displaystyle \lim_{\epsilon \rightarrow 0}(q_1(\epsilon),q_2(\epsilon),...,q_n(\epsilon))=(q_1,...,q_n).$$

In all central configuration of the planar $1+n$ body problem the satellites lie
on a circle centered at the big mass (\cite{llibre1,hall}), i.e. the satellites are coorbital. Since we are interested in central configuration modulus rotations and
homothetic transformations, we can assume that the circle has radius 1 and
that $q_1=(1,0)$.

We exclude collisions in the definition of central configuration and take as coordinates the angles $\theta_i$ between
two consecutive particles. See e.g. \cite{llibre1} for details. In this coordinates the space of configuration is the simplex $$\Delta = \{\theta = (\theta_1,...,\theta_n); \sum_{i=1}^n\theta_i = 2\pi, \theta_i>0, i=1,..,n\}$$ and the equations characterizing the central configurations of the planar 1+n body problem are
\begin{eqnarray}
&&\mu_2f(\theta_1)+\mu_3f(\theta_1+\theta_2)+...+\mu_nf(\theta_1+\theta_2+...+\theta_{n-1})=0, \nonumber \\
&&\mu_3f(\theta_2)+\mu_4f(\theta_2+\theta_3)+...+\mu_1f(\theta_2+\theta_3+...+\theta_{n})=0, \nonumber \\
&&\mu_4f(\theta_3)+\mu_5f(\theta_3+\theta_4)+...+\mu_2f(\theta_3+\theta_4+...+\theta_{n}+\theta_1)=0,\nonumber\\
&&...\label{sistemageral}\\
&&\mu_nf(\theta_{n-1})+...+\mu_{n-2}f(\theta_{n-1}+\theta_n+\theta_1+...+\theta_{n-3})=0, \nonumber\\
&&\mu_1f(\theta_{n})+\mu_2f(\theta_{n}+\theta_{1})+...+\mu_{n-1}f(\theta_{n}+\theta_1+\theta_2+...+\theta_{n-2})=0, \nonumber\\
&&\theta_1+...+\theta_n=2\pi,\nonumber
\end{eqnarray}
where $f(x) = \displaystyle \sin(x)\left(1 - \frac{1}{8|\sin^3(x/2)|} \right).$
\\

\begin{definition}
We say that a solution $(\theta_1,...,\theta_2)$ of the system (\ref{sistemageral}) is a central configuration of the planar $1+n$ body problem associated to the masses $\mu_1,...,\mu_n$.
\end{definition} 

The following results exhibit the main properties of the function $f$. Their prove can be found in \cite{albouy}.

\begin{lemma}\label{lema1}
 The function
 $$f(x)=\sin(x)\left(1 - \frac{1}{8|\sin^3(x/2)|} \right), \ \ x\in (0,2\pi)$$
satisfies:
\begin{enumerate}[i)]
\item $f(\pi/3)=f(\pi)=f(5\pi/3)=0;$
 \item $f(\pi-x)=-f(\pi+x), \forall x \in (0,\pi);$
 \item $\displaystyle f'(x)=\cos(x)+\frac{3+\cos(x)}{16|\sin^3(x/2)|}\geq f'(\pi)=-7/8$, for all $x\in(0,2\pi);$
 \item $f'''(x)> 0$, for all $x\in(0,2\pi);$
\item In $(0,\pi)$ there is a unique critical point $\theta_c$ of $f$ such that $\theta_c>3\pi/5$,  $f'(\theta)>0$ in $(0,\theta_c)$ and $f'(\theta)<0$ in $(\theta_c,\pi).$
\end{enumerate}
\end{lemma}

\begin{lemma}\label{lema2albouy}
Consider four points $t_1^L,t_1^R,t_2^L,t_2^R$ such that $0<t_1^L<t_2^L<\theta_c<t_2^R<t_1^R<2\pi, f(t_1^L)=f(t_1^R)=f_1$ and $f(t_2^L)=f(t_2^R)=f_2$. Then
$t_2^L+t_2^R<t_1^L+t_1^R.$
\end{lemma}

\begin{corollary}\label{corolarioalbouy}
Consider $0<t_1<\theta_c<t_2<2\pi.$ If $f(t_1)\geq f(t_2)$ then $t_1+t_2>2\theta_c>6\pi/5.$
\end{corollary}

\begin{figure}
  \centering
  \includegraphics[width=3in]{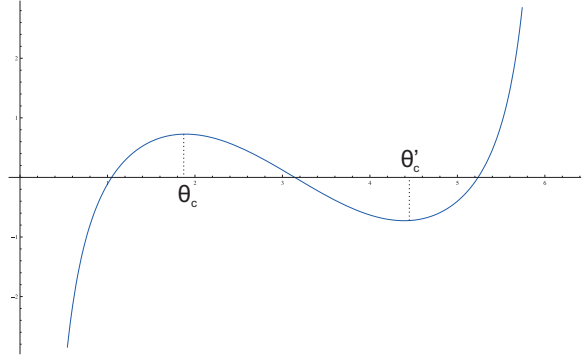}
  \caption{The function $\displaystyle f(x) = \sin(x)\left(1 - \frac{1}{8|\sin^3(x/2)|} \right)$.}
  \label{fig14}
\end{figure}

\section{Main Results}

We now consider the planar problem of $1+4$ bodies, where the four satellites do not necessarily have the same masses. The goal is to find all central configurations with two diametrically opposite satellites. We call them the collinear satellites. See Figure \ref{fig14}. In such way we have a central configuration of the planar $1+2$ body problem in which the satellites and the massive body are collinear. Hence we also get stacked central configurations as introduced by Hampton \cite{hampton}.

\begin{figure}[h]
  \centering
  \includegraphics[width=3in]{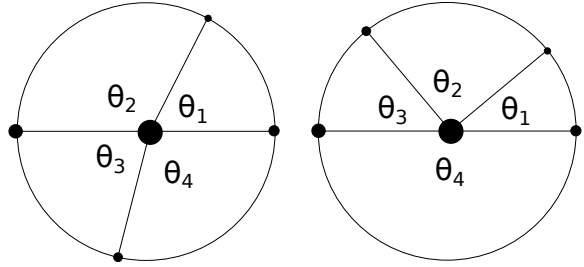}
  \caption{The $1+4$ body problem with two diametrically opposite satellites.}
  \label{fig14}
\end{figure} 
Since $f(x)=-f(2\pi-x)$, in the case $n=4$, the system (\ref{sistemageral}) becomes

\begin{eqnarray}
&&\mu_2f_1+\mu_3f_{12}=\mu_4f_4, \nonumber \\
&&\mu_3f_2+\mu_4f_{23}=\mu_1f_1, \nonumber \\
&&\mu_4f_3+\mu_1f_{34}=\mu_2f_2,\label{system14}\\
&&\mu_1f_4+\mu_2f_{14}=\mu_3f_3,\nonumber \\
&&\theta_1+\theta_2+\theta_3+\theta_4=2\pi,\nonumber
\end{eqnarray}
where $$f_i=f(\theta_i) \mbox{ and } f_{ij}=f(\theta_i+\theta_j).$$

Firstly we consider the case where the two collinear satellites are arranged consecutively in the circle, like the second case in Fig. \ref{fig14}. The next result show that it's impossible to have a central configuaration like that.

\begin{theorem}
 Let $(\theta_1,\theta_2,\theta_3,\theta_4)$ be a central configuration of the planar $1+4$ body problem associated to masses parameters $\mu_1,\mu_2,\mu_3,\mu_4$. Then any two consecutive satellites in conjuction with the massive body can not be collinear, i. e.  $\theta_i\neq \pi$ for all $i=1,2,3,4$.
\end{theorem}

\begin{proof}
 Suppose without loss of generality $\theta_4=\theta_1+\theta_2+\theta_3=\pi$. The system (\ref{system14}) become
 
 \begin{eqnarray}
&&\mu_2f_1=\mu_3f(\pi+\theta_3), \label{t1eq1} \\
&&\mu_3f_2=\mu_4f(\pi+\theta_1)+\mu_1f_1, \label{t1eq2} \\
&&\mu_4f_3+\mu_1f_{34}=\mu_2f_2, \label{t1eq3}\\
&&\mu_3f_3=\mu_2f(\pi+\theta_1), \label{t1eq4}
\end{eqnarray}

Suppose that $f(\pi+\theta_1)\geq 0.$ Then $\pi+\theta_1\geq 5\pi/3$ and we get $\theta_1\geq 2\pi/3$ and $\theta_2+\theta_3\leq \pi/3.$ So $\theta_3<\pi/3$ and $f_3<0$ hence equation (\ref{t1eq4}) is impossible. Therefore $f(\pi+\theta_1)<0$. Analogously $f(\pi+\theta_3)<0.$ 

From (\ref{t1eq1}) and (\ref{t1eq4}), we get $f_1, f_3<0.$ So $\theta_1<\pi/3, \theta_3<\pi/3$ and consequently  $\theta_2>\pi/3$ and $f_2>0$. Hence the right side of (\ref{t1eq2}) is negative and its left side is positive. It's concludes the proof.   
\end{proof}

The remaining results concern the first case in Figure \ref{fig14} namely the collinear satellites are not consecutive in the circle. Firstly we prove that the configuaration is symmetric and the other satellites are identical.

\begin{theorem}\label{theorem1} Let $(\theta_1,\theta_2,\theta_3,\theta_4)$ be a central configuration of the planar $1+4$ body problem associated to masses parameters $\mu_1,\mu_2,\mu_3,\mu_4$. Suppose that the massive body and the satellites with masses $\mu_1$ and $\mu_3$ are collinear, i.e.  $\theta_1+\theta_2=\pi=\theta_3+\theta_4$. Then $\theta_1=\theta_4$ and $\theta_2=\theta_3$, i.e. the configuration is symmetric. Furthermore the other satellites have the same mass $\mu_2=\mu_4$.
\end{theorem}

\begin{proof}
 Since $f(\pi)=0,$ we get
\begin{eqnarray}
\mu_2f_1&=&\mu_4f_4,\label{eq1}\\
\mu_3f_2+\mu_4f_{23}&=&\mu_1f_1,\label{eq2}\\
\mu_4f_3&=&\mu_2f_2,\label{eq3}\\
\mu_1f_1+\mu_2f_{14}&=&\mu_2f_3. \label{eq4}
\end{eqnarray}

Hence, by (\ref{eq1}) and (\ref{eq3})
\begin{equation}\label{eq5}
f_1f_3=f_2f_4.
\end{equation}

From (\ref{eq1}), $f_1=0$ if and only if $f_4=0$. As the only root in $(0,\pi)$ of $f$ is $\pi/3$, then if $f_1=0$ or $f_4=0$ we get $\theta_1=\theta_4=\pi/3$. Analogously if $f_2=0$ or $f_3=0$, then $\theta_2=\theta_3=\pi/3$.
Therefore we will suppose that $f_i\neq 0,i=1,2,3,4$ or equivalently $\theta_i\neq \pi/3,2\pi/3$. 

From (\ref{eq5}) and by hypothesis $\theta_2=\pi-\theta_1, \theta_3=\pi-\theta_4$ we have
\begin{equation}\label{eqlt1lt4}
l(\theta_1)=l(\theta_4),
\end{equation}
where $\displaystyle l(x)=\frac{f(x)}{f(\pi-x)}, x \neq 2\pi/3.$

Note that from (\ref{eq1}) and (\ref{eq3}), $f_1,f_4$ have the same sign as well $f_2,f_3$. Since $f$ is negative in $(0,\pi/3)$ and positive in $(\pi/3,\pi)$ it follows that $0<\theta_1,\theta_4<2\pi/3$ or $2\pi/3<\theta_1,\theta_4<\pi.$ So by (\ref{eqlt1lt4}) is sufficient to show that $l|_{(0,2\pi/3)}$ and $l|_{(2\pi/3,\pi)}$ are injective for obtain $\theta_1=\theta_4$.

We have
$$l'(x)=\frac{f'(x)f(\pi-x)+f(x)f'(\pi-x)}{(f(\pi-x))^2}$$

If $x\in (2\pi/3,\pi)$ then $f'(x)<0,f(\pi-x)<0,f(x)>0$ and $f'(\pi-x)>0$. So $l'(x)>0$ and hence $l|_{(2\pi/3,\pi)}$ is injective.

Let $p(x)=f'(x)f(\pi-x)+f(x)f'(\pi-x)$ be the numerator of $l'(x)$. 
$$p'(x)=f''(x)f(\pi-x)-f(x)f''(\pi-x)$$
If $0<x<\pi/2$ then $\pi/2<x<\pi-x$. Hence $f''(x)<f''(\pi-x)$ and by Corollary \ref{corolarioalbouy} we get $f(x)<f(\pi-x)$. 

As $f''(\pi-x)<0$ we have 
$$f(x)f''(\pi-x)>f(\pi-x)f''(\pi-x)>f(\pi-x)f''(x).$$ 
So $p'(x)<0$ in $(0,\pi/2)$. Since $p(x)=p(\pi-x),$ if $x\in (\pi/2,\pi)$ then $p'(x)>0$. 

Therefore $x=\pi/2$ is the minimum point of $p$ and consequently 
$$p(x)\geq p(\pi/2)=2f'(\pi/2)f(\pi/2)>0.$$
Thus $f'(x)>0$ and hence $\theta_1=\theta_4$ and $\theta_2=\theta_3.$ 

Now if $\mu_2\neq \mu_4,$ by (\ref{eq1}), (\ref{eq2}), (\ref{eq3}) and (\ref{eq4}) we get $f(\theta_1)=f(\theta_2)=f(2\theta_1)=f(2\theta_2)=0.$ But this is impossible because in $(0,2\pi)$ the roots of $f$ are $\pi/3,\pi$ and $5\pi/3$.

\end{proof}

In the next result we count the number of central configuration of this problem in the general case. By last Theorem the configuration are symmetric and two satellites have the same infinitasimal masses. Next Theorem also shows that additional equality of the some infinitesimal masses are equivalent to existence of special configuration as a square and a kite.  
\begin{theorem}\label{theorem2}
Let $(\theta_1,\theta_2,\theta_3,\theta_4)$ be a central configuration of the planar $1+4$ body problem associated to masses parameters $\mu_1,\mu_2,\mu_3,\mu_4$. Suppose that $\theta_1+\theta_2=\theta_3+\theta_4=\pi$. Then for all values of the masses parameters the number of classes of central configuration is one, two or three. Moreover the square $(\pi/2,\pi/2,\pi/2,\pi/2)$ is a central configuration if and only if $\mu_3=\mu_1$ and the kite $(2\pi/3,\pi/3,\pi/3,2\pi/3)$ is a central configuration if and only if $\mu_1=\mu_2$ 
\end{theorem}

\begin{proof}

 By Theorem \ref{theorem1} we know that $\mu_2=\mu_4, \theta_1=\theta_4,\theta_2=\theta_3=\pi-\theta_1.$ So the equations (\ref{eq1}) and (\ref{eq3}) are redundant and (\ref{eq2}) and (\ref{eq4}) are equivalent to each other and they become
\begin{equation}
f(\theta_1)+\frac{\mu_2}{\mu_1}f(2\theta_1)+\frac{\mu_3}{\mu_1}f(\pi+\theta_1)=0.
\end{equation} 

We must to say how many roots in $(0,\pi)$ has the function
$$g(x)=f(x)+\frac{\mu_2}{\mu_1}f(2x)+\frac{\mu_3}{\mu_1}f(\pi+x).$$

It's easy to see that $g(x)\rightarrow -\infty$ if $x \rightarrow 0^{+}$ and $g(x)\rightarrow +\infty$ if $x \rightarrow \pi^{-}$. So for any $\mu_1,\mu_2,\mu_3>0$ there is at least one solution of $g(x)=0$ in $(0,\pi)$. Moreover, since $\displaystyle g'''(x)=f'''(x)+8\frac{\mu_2}{\mu_1}f'''(2x)+\frac{\mu_3}{\mu_1}f'''(\pi+x)>0$, there are at most 3 roots of $g$ in $(0,\pi)$.

Observe that, by Lemma \ref{lema1},  
\begin{eqnarray*}
g(\pi/2)&=&f(\pi/2)+\frac{\mu_2}{\mu_1}f(\pi)+\frac{\mu_3}{\mu_1}f(\pi+\pi/2)\\
&=&f(\pi/2)+\frac{\mu_3}{\mu_1}f(\pi+\pi/2)\\
&=&f(\pi/2)-\frac{\mu_3}{\mu_1}f(\pi-\pi/2)\\
&=&f(\pi/2)\left(1-\frac{\mu_3}{\mu_1}\right).
\end{eqnarray*}

Hence $g(\pi/2)=0$ if and only if $\mu_3=\mu_1$. 
Likewise 
\begin{eqnarray*}
g(2\pi/3)&=&f(2\pi/3)+\frac{\mu_2}{\mu_1}f(4\pi/3)+\frac{\mu_3}{\mu_1}f(5\pi/3)\\
&=&f(2\pi/3)-\frac{\mu_2}{\mu_1}f(2\pi/3)\\
&=&f(2\pi/3)\left(1-\frac{\mu_2}{\mu_1}\right).
\end{eqnarray*}

So $g(2\pi/3)=0$ if and only if $\mu_2=\mu_1$.

\end{proof}

The special case where the collinear satellites has the same mass is completely treated in the next Theorem. We have now two parameters of masses, $\mu_1=\mu_3$ and $\mu_2=\mu_4$. The ratio $\mu_2/\mu_1$ provides a parameter for bifurcation in the number of central configuration of this problem. We give it's exact value in next Theorem.

\begin{theorem}
Let $(\theta_1,\theta_2,\theta_3,\theta_4)$ be a central configuration of the planar $1+4$ body problem associated to masses parameters $\mu_1,\mu_2,\mu_3,\mu_4$. Assume that $\theta_1=\theta_4, \theta_2=\theta_3=\pi-\theta_1$. Suppose that $\mu_1=\mu_3$. If $\mu_2/\mu_1\leq\frac{3\sqrt{2}}{7}$ there is a unique central configuration: the square $\theta_i=\pi/2, i=1,2,3,4$. If $\mu_2/\mu_1>\frac{3\sqrt{2}}{7}$ there are two central configurations: the square and the kite $(\theta_1,\pi-\theta_1,\pi-\theta_1,\theta_1)$ where $\theta_1\in (\pi/6,\pi/2)$. In fact the function mapping $\mu_2/\mu_1\in (\frac{3\sqrt{2}}{7},+\infty)$ to $\theta_1\in (\pi/6,\pi/2)$ in the kite configuration is a bijective function.
\end{theorem}

\begin{proof}
 Let $a=\mu_2/\mu_1.$ The central configurations are determinated by equation
 \begin{equation}\label{eq_theta1}
  f(\theta_1)+af(2\theta_1)+f(\pi+\theta_1)=0, 
 \end{equation}
with $\theta_1\in (0,\pi).$ Again, consider the function $g: (0,\pi)\rightarrow \mathbb{R}$ given by
$$g(x)=f(x)+af(2x)+f(\pi+x).$$

Since $g'''(x)>0,$ then $g''(x)=f''(x)+4af''(2x)+f''(\pi+x)$ is increasing. Moreover $g''(\pi/2)=0$ because $f''(\pi-x)=-f''(\pi+x)$, so we have that $x=\pi/2$ is the minimal value of $g'(x)$ in $(0,\pi)$.

We obtain
\begin{eqnarray}
 g'(\pi/2)&=&f'(\pi/2)+2af'(\pi)+f'(3\pi/2)\\
 &=&2f'(\pi/2)+2af'(\pi)\\
 &=&2\left(\frac{3\sqrt{2}}{8}-\frac{7a}{8}\right)\\
 &=&\frac{1}{4}(3\sqrt{2}-7a)
\end{eqnarray}

So if $a\leq \frac{3\sqrt{2}}{7}$ then $g'(x)\geq g'(\pi/2)\geq 0$ and consequently $g$ have only one root in $(0,\pi)$. As $x=\pi/2$ is always a root of $g$, then we have only the square $(\pi/2.\pi/2,\pi/2,\pi/2)$.

If $a> \frac{3\sqrt{2}}{7}$ then $g'(\pi/2)<0$, so $g$ has three roots in $(0,\pi)$, namely $\theta_1^*, \pi/2$ and $\theta_1^{**}$. Since $g(\pi-x)=-g(x)$ then $\theta_1^{**}=\pi-\theta_1^*$ and that roots correnpond to the same configuration, the kite $(\theta_1,\pi-\theta_1,\pi-\theta_1,\theta_1)$.

Now we will look for the kite central configuration $(\theta_1,\pi-\theta_1,\pi-\theta_1,\theta_1)$. $\theta_1$ agree with (\ref{eq_theta1}). By the symmetry we can consider $\theta_1 \in (0,\pi/2)$. If $\theta_1\neq \pi/6$ that equation is equivalently to
\begin{equation}\label{eq_a}
 a=\frac{-f(\pi+\theta_1)-f(\theta_1)}{f(2\theta_1)}.
\end{equation}

If $\theta_1<\pi/6$ the right side is negative and the left one is positive. Also observe that $\theta_1=\pi/6$ does not agree with (\ref{eq_theta1}). So we have no solution in $(0,\pi/6]$. Futhermore the right side of (\ref{eq_a}) satisfies
$$\lim_{\theta_1\rightarrow \pi/6^{+}}\frac{-f(\pi+\theta_1)-f(\theta_1)}{f(2\theta_1)}=+\infty$$
and
$$\lim_{\theta_1\rightarrow \pi/2^{-}}\frac{-f(\pi+\theta_1)-f(\theta_1)}{f(2\theta_1)}=\lim_{\theta_1\rightarrow \pi/2^{-}}\frac{-f'(\pi+\theta_1)-f'(\theta_1)}{2f'(2\theta_1)}=\frac{-f'(\pi+\pi/2)-f'(\pi/2)}{2f'(\pi)}=\frac{3\sqrt{2}}{7}.$$

Consider then the surjective function  $h:(\pi/6,\pi/2)\rightarrow (\frac{3\sqrt{2}}{7},+\infty)$ given by
$$h(x)=\frac{-f(\pi+x)-f(x)}{f(2x)}.$$

We claim that $h$ is one-to-one too. In fact it is a decreasing function. To see that we derivate $h$ and obtain
$$(f(2x))^2h'(x)=2f'(2x)(f(x)+f(\pi+x))-f(2x)(f'(x)+f'(\pi+x)).$$

The derivate of the right side of the above equation is given by
$$4f''(2x)(f(x)+f(\pi+x))-f(2x)(f''(x)+f''(\pi+x))=$$
$$4f''(2x)(f(x)-f(\pi-x))-f(2x)(f''(x)-f''(\pi-x)).$$

The equality follows by Lemma \ref{lema1}. We claim that the above expression is positive. In fact $x<\pi/2$, then $\pi-x>x.$ It follows that $f''(\pi-x)>f''(x)$ because $f'''(x)>0$. By Corollary \ref{corolarioalbouy} we have $f(\pi-x)>f(x)$. Moreover $f''(2x)<0$ and $f(2x)>0$ as $x\in (\pi/6,\pi/2).$ So the statement is true. This shows that the expression for $(f(2x))^2h'(x)$ takes its maximum value in limit case $x\rightarrow \pi/2$. But
$$\lim_{x\rightarrow \pi/2}(f(2x))^2h'(x)=\lim_{x\rightarrow \pi/2}(2f'(2x)(f(x)+f(\pi+x))-f(2x)(f'(x)+f'(\pi+x)))=0.$$

Therefore $(f(2x))^2h'(x)<0$ if $x\in (\pi/6,\pi/2)$ and thus $h$ is decreasing in the same interval. The Theorem follows.

\end{proof}

\section{Conclusions}

We studied the relative equilibria of the planar $1+4$ body problem in the case where two sattelites are diametrically opposite in the circle centered in the massive body. So these sattelites and the big mass are collinear. We show that all central configurations are symmetric kites and a square and the other two sattelites have the same mass. Moreover we prove that there are one, two or three such configurations and only in the case where the collinear satellites have different masses is possible to have three central configurations. If the collinear satellites have equal masses $\mu_1$ and the others satellites have masses $\mu_2$ we gave all relative equilibria and we calculated the value of ratio $\mu_2/\mu_1$ in which it provides the bifurcation from one to two central configurations.

The collinear configuration with two satellites diametrically opposite is a central configuration of  the planar $1+2$ body problem. So our aproach is a study of stacked central configurations too. In this way our results show that adding two new sattelites in a collinear $1+2$ configuration we get  new central configuration if and only if the two new satellites has the same masses, they are put symmetrically and the smaller angle between them and the line of the collinear satellites varies from $\pi/6$ to $\pi/2$. 



\medskip
\medskip


\begin{thebibliography}{99}


\bibitem{albouy}
     \newblock A. Albouy and Y. Fu, 
     \newblock \emph{Relative equilibria of four identical satellites},
     \newblock Proc. R. Soc. Lond. Ser. A Math. Phys. Eng. Sci., \textbf{465} (2009), 2633--2645.

\bibitem{llibre1}
     \newblock J. Casasayas, J. Llibre and A. Nunes,
     \newblock \emph{Central configurations of the planar $1+n$ body problem},
     \newblock Celestial Mech. Dynam. Astronom, \textbf{60} (1994), 273--288.

\bibitem{llibre3} 
    \newblock M. Corbera, J. Cors and J. Llibre, 
    \newblock \emph{On the central configurations of the planar $1+3$ body problem},
    \newblock Celestial Mech. Dynam. Astronom, \textbf{109} (2011), 27--43.

\bibitem{llibre2} 
     \newblock J. Cors, J. Llibre and M. Oll\'e,
     \newblock \emph{Central configurations of the planar coorbital satellite problem},
     \newblock Celestial Mech. Dynam. Astronom, \textbf{89} (2004), 319--342.
     
\bibitem{hagihara} 
    \newblock Y. Hagihara,
    \newblock  \emph{Celestial mechanics. Volume I: Dynamical principles and transformation theory},
    \newblock The MIT Press, Cambridge, Mass.-London (1970).
    
\bibitem{hall}
\newblock G. Hall,
\newblock \emph{Central configuration in the planar $1 + n$ body problem},
\newblock  preprint, (1988).

\bibitem{hampton} 
     \newblock M. Hampton,
     \newblock \emph{Stacked central configurations: new examples in the planar five-body problem},
     \newblock Nonlinearity, \textbf{18} (2005), 2299--2304.
     

\bibitem{mello} 
     \newblock J. Llibre and L. Mello,
     \newblock \emph{New central configurations for the planar 5-body problem},
     \newblock Celestial Mech. Dynam. Astronom, \textbf{100} (2008), 141--149.
     
\bibitem{maxwell}
     \newblock J. Maxwell,
     \newblock \emph{On the Stability of Motion of Saturn's Rings},
     \newblock Macmillan $\&$ Co., London, (1985).
     
\bibitem{oliveira}
     \newblock A. Oliveira and H. Cabral,
     \newblock \emph{On stacked central configurations of the planar coorbital satellites problem},
     \newblock Discrete and Continuous Dynamical Systems, \textbf{32} (2012), 3715--3732.

\bibitem{sicardy}
     \newblock S. Renner and B. Sicardy,
     \newblock \emph{Stationary configurations for co-orbital satellites with small arbitrary masses},
     \newblock Celestial Mech. Dynam. Astronom. \textbf{88}, (2004) 397--414.
     
\bibitem{saari}
     \newblock D. Saari,
     \newblock \emph{On the role and the properties of $n$-body central configurations},
     \newblock Celestial Mech., \textbf{21}  (1980). 9--20  
     
\bibitem{saari2}
     \newblock D. Saari,
     \newblock \emph{Collisions, rings, and other Newtonian $N$-body problems},
     \newblock CBMS regional conference series in mathematics, 104, American Mathematical Society  (2005).


\bibitem{wintner}
	\newblock A. Wintner,
	\newblock \emph{The Analytical Foundations of Celestial Mechanics}, 	\newblock Princeton University Press, (1941), 1--43.


\end{thebibliography}
\end{document}